\documentclass[12pt]{amsart}
\usepackage[margin=3cm]{geometry}

\usepackage[utf8]{inputenc}
\usepackage[english]{babel}
\usepackage{amsmath,amssymb,amsthm}
\usepackage{enumerate,mathrsfs}
\usepackage{color}

\newtheorem{theorem}{Theorem}[section]

\newtheorem{lemma}[theorem]{Lemma}
\newtheorem{proposition}[theorem]{Proposition}
\newtheorem{definition}[theorem]{Definition}

\theoremstyle{definition}

\newtheorem{remark}[theorem]{Remark}

\newcommand{\PG}{\mathrm{PG}}
\newcommand{\K}{\mathbb{K}}

\title[Group-labeled light dual multinets]{Group-labeled light dual multinets in the projective plane (with Appendix)}
\author{G\'abor Korchm\'aros}
\address{Dipartimento di Matematica, Informatica ed Economia\\
	Universit\`a della Basilicata\\
	Contrada Macchia Romana\\
	85100 Potenza, Italy}
\email{gabor.korchmaros@unibas.it}
\author{G\'abor P. Nagy}
\address{Department of Algebra \\
	Budapest University of Technology and Economics\\
	Egry J\'ozsef utca 1\\
	H-1111 Budapest, Hungary}
\address{Bolyai Institute \\
	University of Szeged \\
	Aradi v\'ertan\'uk tere 1\\
	H-6720 Szeged, Hungary}
\email{nagyg@math.bme.hu}
\thanks{Research supported by NKFIH-OTKA Grants 114614 and 119687.}
\date{Version 30/09/2017.}

\begin{document}

\begin{abstract}
In this paper we investigate light dual multinets labeled by a finite group in the projective plane $PG(2,\mathbb{K})$  defined over a field $\mathbb{K}$. We present two classes of new examples. Moreover, under some conditions on the characteristic $\mathbb{K}$, we classify group-labeled light dual multinets with lines of length least $9$.
\end{abstract}

\maketitle

\section{Introduction}
Let $\PG(2,\mathbb{K})$ be the projective plane coordinatized by an algebraically closed field $\mathbb{K}$ of characteristic $p\geq  0$, and let $Q$ be a finite set equipped with a binary operation $x\cdot y$. A possibility of linking the algebraic structure of $Q$ to point-line incidences in $\PG(2,\mathbb{K})$ arises when $Q$ may be embedded in $\PG(2,\mathbb{K})$ in the sense that there exist three maps from $Q$ into $\PG(2,\mathbb{K})$, say $\alpha_1,\alpha_2,\alpha_3$, such that if $x\cdot y=z$ then $\alpha_1(x)$, $\alpha_2(y)$ and $\alpha_3(z)$ are three collinear points in $\PG(2,\mathbb{K})$. Here, the point-sets $\Lambda_i=\alpha_i(Q)$ with $i=1,2,3$ are called components and assumed to be pairwise disjoint. What we can derive from the algebraic structure of $Q$ in this way is a collection of properties of the incidence structure cut out on $\Lambda_1\cup\Lambda_2\cup \Lambda_3$ by lines meeting each component. In most cases, such an incidence structure is rather involved and not very interesting, but there are significant exceptional cases, some of which arise from Algebraic geometry.

In the simplest and perhaps nicest particular case,
each line of $\PG(2,\mathbb{K})$ meeting two distinct components meets each component in an exactly one point; in particular $|\Lambda_1|=|\Lambda_2|=|\Lambda_3|=n$. Therefore, the incidence structure is a dual $3$-net of order $n$, that is a $3$-net of order $n$ embedded in the dual plane of $PG(2,\mathbb{K})$. In this case, $Q$ must be a quasigroup and the maps $\alpha_i$ are bijections.
Conversely, if $\Lambda=(\Lambda_1,\Lambda_2,\Lambda_3)$ is any dual $3$-net in $\PG(2,\mathbb{K})$  then $\Lambda$ is a quasigroup-labeled dual $3$-net. In fact, there exists a quasigroup $Q$  of order $n$ together with a labeling of the sets $\Lambda_1$, $\Lambda_2$, $\Lambda_3$ by elements of $Q$ such that points labeled by $x,y,z$ are collinear if and only if $x\cdot y=z$ holds.

If $Q$ is a group, which is an important particular case in our context, a fairly complete classification of all dual $3$-nets is available for $p=0$ or $p>n$; see \cite{knp_3}: Apart from four groups of smaller orders ($n=8,12,24,60$), a group-labeled dual $3$-net is either algebraic, or of tetrahedron type. An algebraic dual $3$-net has its components lying on a plane cubic $\Gamma$ in $PG(2,\mathbb{K})$ and is labeled by an abelian group. More precisely, the group is cyclic when either $\Gamma$ splits into three non-concurrent lines (triangle type dual $3$-net), or into an irreducible conic and a line (conic-line type), otherwise the group is a subgroup of the Jacobian variety of $\Gamma$ and is either cyclic or the direct product of two cyclic groups. A tetrahedron type dual $3$-net is labeled by a dihedral group, and it can be viewed  in the projective space  $\PG(3,\mathbb{K})$ as the projection (from a point $P$ on a plane $\pi=\PG(2,\mathbb{K})$) of six point sets of size $n$ lying on the sides of a tetrahedron  whenever these point sets, named $\lambda_i,\lambda_i'$ with $i=1,2,3$, have the following two properties: $\lambda_i$ and $\lambda_i'$ lie on opposite sides, and $\lambda_i\cup\lambda_i'$ are the components of a dual $3$-net of order $2n$ in $\PG(3,\mathbb{K})$. In fact, if the center $P$ of the projection is chosen outside the faces of the tetrahedron and $\Lambda_i=\alpha_i(\lambda_i)\cup \alpha_i(\lambda_i')$ for $i=1,2,3$ then $(\Lambda_1,\Lambda_2,\Lambda_3)$ is a dual $3$-net of order $2n$.

A more general, yet interesting case for applications occurs when $Q$ is a quasigroup and each $\alpha_i$ is a bijection. The arising geometric configuration $(\Lambda_1,\Lambda_2,\Lambda_3)$ in $\PG(2,\mathbb{K})$, called \emph{light dual multinet}, has still some regularity, such as $|\ell\cap \Lambda_1|=|\ell\cap \Lambda_2|=|\ell\cap \Lambda_3|=r$ for any line $\ell$ in $\PG(2,\mathbb{K})$ where $r$ depends on $\ell$ and is called the length of $\ell$. Obviously, any dual $3$-net is a light dual multinet, but the converse is not true. A counterexample is obtained if the center of the projection in the above construction is chosen on a face (but a side) of the tetrahedron. The resulting light multinet of tetrahedron type is not a dual $3$-net since some lines in $\PG(2,\mathbb{K})$ meet each component in $n$ points; see \cite{Bartz2013,BartzYuz2014}. If $(\Lambda_1,\Lambda_2,\Lambda_3)$ is a light dual multinet other than dual $3$-nets, then some line in $\PG(2,\mathbb{K})$  meets each component in exactly $r>1$ points. This shows that if $(\Lambda_1,\Lambda_2,\Lambda_3)$ is algebraic, then it is contained in a reducible plane cubic, and hence is either of pencil type, or triangular, or of conic-line type according as the cubic splits into three lines $\ell_1,\ell_2, \ell_3$, either concurrent or not, or in an irreducible conic $C$ plus a  line $\ell$.

In this paper, we present two new families of light dual multinets of order $n$. Those in the first family have order divisible by $3$ and are of triangular type with $\ell_1,\ell_2,\ell_3$ of length $n/3$. Light dual multinets in the second family have even order and they are of conic-line type with $\ell$ of length $n/2$.

As we pointed out in our papers, if $p>0$ is small compared to the size of $Q$, many examples of dual $3$-nets and light multinets arise from configurations of finite projective subplanes.
Under the usual conditions on the characteristic of the underlying field, our main result describes group-labeled light dual multinets with a line of length at least $9$.

\begin{theorem}\label{th:main}
Let $\PG(2,\K)$ be the projective plane coordinatized by an algebraically closed field $\K$ of characteristic $p\geq 0$. Let $G$ be a group of order $n$ and assume $p=0$ or $p>n$. Let $\Lambda=(\Lambda_1,\Lambda_2,\Lambda_3)$ be a light dual multinet in $\PG(2,\K)$, labeled by $G$. If $\Lambda$ has a line of length at least $9$, then exactly one of the following cases occurs:
\begin{enumerate}[(i)]
\item[\rm(i)] $\Lambda$ is contained in a line.
\item[\rm(ii)] $\Lambda$ has a line of length $3$.
\item[\rm(iii)] $\Lambda$ is either triangular, or of tetrahedron type, or of conic-line type.
\end{enumerate}
\end{theorem}

The paper is organized as follows. In Section 2 we put down the details on the labeling of light dual multinets and recall some facts of dual nets. In Section 3 we present our new examples of light dual multinets. In Section 4 we study light dual multinets labeled by the cyclic group. In Proposition \ref{pr:cyclic_alg_net} we show that such multinets are algebraic. In Section 5 we prove a series of lemmas on the non-extendability of the known light dual multinet constructions. Finally, in Section 6 we prove Theorem \ref{th:main} using the geometric results obtained so far, and some theorems of Herstein and Horo\v{s}evski\u{\i} on abstract finite groups.

For a thorough discussion on the applications of finite multinets in Algebraic geometry, especially in the study of completely reducible fibers of pencils of hypersurfaces, and also in complex line arrangements and resonance theory; see \cite{bkm,bkn,fy2007,knp_3,knp_k,miq,per,ys2004,ys2007}.

%
%
%

\section{Notation, terminology and background on dual multinets}

\subsection{Light dual multinets}

A \emph{quasigroup} is a set $Q$ endowed with a binary operation $x\cdot y$ such that the equation $x\cdot y=z$ can be uniquely resolved whenever two of the three values $x,y,z\in Q$ are known. Quasigroups with a multiplicative unit element are called \emph{loops.} Two quasigroups $Q(\cdot)$ and $R(\circ)$ are \emph{isotopic} if $\alpha(x)\circ\beta(y)=\gamma(x\cdot y)$ holds for some bijective maps $\alpha,\beta,\gamma:\, Q\rightarrow R$. If $Q(\cdot)$ and $R(\circ)$ are isotopic groups and the are isomorphic.

\begin{definition}
Let $\K$ be an algebraically closed field and $(Q,\cdot)$ be a finite quasigroup of order $n$. A \emph{dual multinet labeled by $Q$} is a triple $(\alpha_1,\alpha_2,\alpha_3)$ of maps $Q\to \PG(2,\K)$ such that for any $x,y\in Q$, the points $\alpha_1(x), \alpha_2(y)$ and $\alpha_3(xy)$ are collinear. The dual multinet is said to be \emph{light} if the maps $\alpha_1,\alpha_2,\alpha_3$ are injective.
\end{definition}
If it is also requested that for any $x,y\in Q$, the points $\alpha_1(x), \alpha_2(y)$ and $\alpha_3(z)$ are only collinear when  $z=x\cdot y$ then the light dual multinet is a dual $3$-net.

The point sets $\Lambda_i=\alpha_i(Q)$ are the \emph{components} of the dual multinet and we view the map $\alpha_i$ as a \emph{labeling} of the component $\Lambda_i$. However, if a dual multinet is not a dual $3$-net it may happen that $(\Lambda_1,\Lambda_2,\Lambda_3)$ depends on the labelings. Therefore, $(\Lambda_1,\Lambda_2,\Lambda_3)$ stands for the dual multinet arising from $Q$ whenever the labelings are clear from the context.
Bearing this in mind, we associate the incident structure to a dual multinet whose points are those of the components and whose lines are cut out by the lines meeting all three components. Also, each such line of $PG(2,\mathbb{K})$ is said \emph{to belong to the dual multinet}. For instance, for $p_1=\alpha_1(x)$, $p_2=\alpha_2(y)$, $p_3=\alpha_3(z)$, the lines $p_1p_2$, $p_1p_3$, $p_2p_3$ hit the third component in the points $\alpha_3(xy)$, $\alpha_2(x\backslash z)$, $\alpha_1(z/y)$, respectively, and hence the belong to the multiset.

\begin{lemma} \label{lm:lengthdef}
Let $\ell$ be a line belonging to the dual multinet $(\Lambda_1,\Lambda_2,\Lambda_3)$. Put $S_i=\{ x\in Q \mid \alpha_i(x) \in \ell \}$. Then, $S_1\cdot S_2\subseteq S_3$, $S_3/S_2\subseteq S_1$, $S_1\backslash S_3\subseteq S_2$. In particular, $|S_1|=|S_2|=|S_3|$. The latter integer is called the \emph{length} of $\ell$ w.r.t. the given dual multinet. \qed
\end{lemma}

\begin{remark}
In \cite[Section 3]{per}, the authors showed (in the primal setting) that $|\ell \cap \Lambda_i|$ is independent of $i$.
\end{remark}

\subsection{Relabeling} Let $\ell$ be a line belonging to the dual multinet $(\Lambda_1,\Lambda_2,\Lambda_3)$ and assume $\alpha_1(u),\alpha_2(v),\alpha_3(uv) \in \ell$ for some $u,v\in Q$. Define the map $\beta_i:Q\to \Lambda_i$ ($i=1,2,3$) by
\[\beta_1(x)=\alpha_1(x/v), \qquad \beta_2(y)=\alpha_2(u\backslash y), \qquad \beta_3(z)=\alpha_3(z).\]
Then, the $\beta_i$'s are another labelings of $(\Lambda_1,\Lambda_2,\Lambda_3)$ by $Q$ with respect to the operation
\[ x\circ y= (x/v)(u\backslash y).\]
Indeed, $x\circ y=z$  if and only if $(x/v)(u\backslash y)=z$, and, $\beta_1(x),\beta_2(y),\beta_3(z)$ are collinear if and only if $\alpha_1(x/v), \alpha_2(u\backslash y),\alpha_3(z)$ are collinear. The quasigroup $(Q,\circ)$ is a loop with unit element $e=uv$; it is also called the principal loop isotope of $(Q,\cdot)$. The elements $\beta_1(e)=\alpha_1(u), \beta_2(e)=\alpha_2(v)$ and $\beta_3(e)=\alpha_3(uv)$ are on $\ell$. This proves the following

\begin{lemma} \label{lm:relabeling}
Assume that the dual multinet $\Lambda=(\Lambda_1,\Lambda_2,\Lambda_3)$ is labeled by the quasigroup $(Q,\cdot)$ with labeling $(\alpha_1,\alpha_2,\alpha_3)$. Let $\ell$ be a line belonging to $\Lambda$. Then there is a loop isotope $(Q,\circ,e)$ of $(Q,\cdot)$ and a corresponding relabeling $\beta_i:Q\to \Lambda_i$ ($i=1,2,3$) of $\Lambda$ such that $\beta_i(e) \in \ell$. \qed
\end{lemma}

Lemmas \ref{lm:lengthdef} and \ref{lm:relabeling} imply

\begin{lemma}
Let $\Lambda=(\Lambda_1,\Lambda_2,\Lambda_3)$ be a dual multinet labeled by the quasigroup $Q$ with labeling $(\alpha_1,\alpha_2,\alpha_3)$. Let $\ell$ be a line belonging to $\Lambda$. Up to isotopy and relabeling, $Q$ is a loop which has a subloop $S$ such that $\alpha_i(S)=\ell \cap \Lambda_i$.
\end{lemma}

 A dual multinet $(\Lambda_1,\Lambda_2,\Lambda_3)$ is \emph{algebraic,} if all points of its components are contained in a cubic curve $\Gamma$, not necessarily irreducible or nonsingular. If $\Gamma$ splits into  a line and an irreducible conic, then the dual multinet is said to be of \emph{conic-line type.} If $\Gamma$ splits into three lines, then the dual multinet is said to be of \emph{pencil} or of \emph{triangle type,} depending upon whether the lines are concurrent or not. This terminology was introduced for dual $3$-nets and it is also meaningful and useful for dual multinets.

\subsection{Dual $3$-nets} As we have already pointed out, dual $3$-nets are particular light dual multinets which have been investigated in a series of papers. Some of their properties are collected in the following proposition.

\begin{proposition} \label{pr:3net-properties}
Let $\Lambda=(\Lambda_1,\Lambda_2,\Lambda_3)$ be a dual $3$-net labeled by the quasigroup $Q$ with labeling $(\alpha_1,\alpha_2,\alpha_3)$ in the projective plane $\PG(2,\K)$. Let $n$ be the order of $Q$ and $p$ the characteristic of $\K$.
\begin{enumerate}[(i)]
\item (\cite[Section 4.2]{knp_3} and {\cite[Proposition 3.3]{ys2004}}) Let $\Lambda$ be of triangle type and denote the lines containing $\Lambda_i$ by $\ell_i$, $i=1,2,3$. Then $Q$ is a cyclic group of order $n$ and there is an abelian collineation group of order $n^2$ leaving the $\Lambda_i$'s invariant.
\item (\cite[Lemma 3]{knp_3} and {\cite[Proposition 4.3]{ys2004}}) Let $\Lambda$ be of pencil type. Then $Q$ is a subgroup of the additive group of $\mathbb{K}$. In particular, $n$ is a power of $p$.
\item (\cite[Section 4.3]{knp_3} and {\cite[Proposition 5.6]{ys2004}}) Let $\Lambda$ be of conic-line type such that $\Lambda_1$ is contained in the line $\ell$ and $\Lambda_2\cup \Lambda_3$ is contained in the irreducible conic $C$. Then $Q$ is a cyclic group of order $n$. Moreover, there is a generating element $h\in Q$ and a collineation $\varphi$ of order $n$ preserving the $\Lambda_i$'s, such that for all $x\in Q$, $\varphi$ maps the point $\alpha_3(x)$ to $\alpha_3(xh)$.
\item (\cite[Theorem 5.1]{bkm}) If $p=0$ or $p>n$ and $\Lambda_1$ is contained in a line then $\Lambda_2\cup \Lambda_3$ is contained in a conic. $Q$ is a cyclic group of order $n$.
\item (\cite[Proposition 18]{knp_3} and {\cite[Theorem 5.3]{ys2004}}) Dual $3$-nets labeled by a cyclic group are algebraic. \qed
\end{enumerate}
\end{proposition}

\begin{proposition}[{\cite[Proposition 10]{knp_3}}] \label{pr:triangles}
If $(\Gamma_1,\Gamma_2,\Gamma_3)$ and $(\Sigma_1,\Sigma_2,\Sigma_3)$ are triangular dual $3$-nets of order $n\geq 3$ such that $\Gamma_1=\Sigma_1$, then the associated triangles share the vertices on their common side. \qed
\end{proposition}

\section{Examples of light dual multinets}

In this section, we construct two new classes of light dual multinets. The new examples are algebraic as they are contained in a (reducible) cubic curve. We also present the construction by Bartz and Yuzvinsky \cite{BartzYuz2014}, which is obtained as a projection of a dual $3$-net in $\PG(3,\K)$. Finally, we show an example of a light dual multinet of order 18, which cannot be labeled by a group and which has three lines of length $3$.

\subsection{Triangle type light dual multinet} Let $m$ be a positive integer, $n=3m$, $\xi$ a root of unity of order $3m$ in $\K$, $G=\langle \xi \rangle$, and $H=\langle \xi^3 \rangle$ the unique subgroup of index $3$ in $G$. Define the maps $f_1,f_2,f_3: G\to \PG(2,\K)$ by
\[f_1(x)=(0,1,x), \qquad f_2(y)=(y,0,1), \qquad f_3(z)=(z,-1,0).\]
Then $(f_1,f_2,f_3)$ labels a dual $3$-net of triangle type: $xy=z$ if and only if $f_1(x),f_2(y)$ and $f_3(z)$ are collinear. Moreover, the components $f_i(G)$ are contained in the lines $\ell_i:X_i=0$ for $i=1,2,3$.

We define three new maps $\alpha_1,\alpha_2,\alpha_3:G\to \PG(2,\K)$ by
\[\alpha_1: \left\{ \begin{array}{l} x\mapsto f_1(x) \\ x\xi \mapsto f_2(x\xi^{-1}) \\ x\xi^2 \mapsto f_3(x^{-1}\xi^2) \end{array} \right.
\quad
\alpha_2: \left\{ \begin{array}{l} y\mapsto f_1(y\xi) \\ y\xi \mapsto f_2(y) \\ y\xi^2 \mapsto f_3(y^{-1}\xi) \end{array} \right.
\quad
\alpha_3: \left\{ \begin{array}{l} z\mapsto f_1(z^{-1}\xi^5) \\ z\xi \mapsto f_3(z) \\ z\xi^2 \mapsto f_2(z^{-1}\xi) \end{array} \right.
\]
where $x,y,z \in H$. It is straightforward to check that $(\alpha_1,\alpha_2,\alpha_3)$ is a light dual multinet  labeled by the cyclic group $G$ of order $n$. For example, $(x\xi)(y\xi^2)=xy \xi^3$, and $\alpha_1(x\xi)=f_2(x\xi^{-1})$, $\alpha_2(y\xi^2)=f_3(y^{-1}\xi)$ and $\alpha_3(xy \xi^3)=f_1(x^{-1}y^{-1}\xi^2)$ are collinear by $(x^{-1}y^{-1}\xi^2)(x\xi^{-1})=y^{-1}\xi$. The light dual multinet $(\alpha_1(G),\alpha_2(G),\alpha_3(G))$ has three lines of length $m=n/3$.

Dualizing $(\alpha_1(G),\alpha_2(G),\alpha_3(G))$ we obtain a light multinet, which corresponds to the following completely reducible polynomials of degree $n=3m$:
\begin{align*}
F_1(X,Y,Z)&=(X^m-Y^m)(Z^m-\omega^2 X^m)(Y^m-\omega Z^m)\\
F_2(X,Y,Z)&=(X^m-\omega Y^m)(Z^m-X^m)(Y^m-\omega^2 Z^m)\\
F_3(X,Y,Z)&=(X^m-\omega^2 Y^m)(Z^m-\omega X^n)(Y^m-Z^m),
\end{align*}
where $\omega$ is a 3rd root of unity in $\K$. Indeed, $F_1+F_2+F_3=0$.

\subsection{Conic-line type light dual multinets} Define the line $\ell:X_3=0$ and the conic  $C:X_1X_2-X_3^2$ in the homogeneous coordinate system $(X_1,X_2,X_3)$. Let $f_1:\K^*\to \ell$ and $f_2:\K^*\to C$ be functions given by
\[f_1(u)=(u,-1,0), \qquad f_2(u)=(u,u^{-1},1).\]
Then, for $x\neq y$, the points $f_2(x)$, $f_2(y)$ and $f_1(z)$ are collinear if and only if $xy=z$.

As in the previous example, let $m$ be a positive integer, $n=2m$, $\xi$ a root of unity of order $3m$ in $\K$, $G=\langle \xi \rangle$, and $H=\langle \xi^3 \rangle$ the unique subgroup of index $3$ in $G$. Let $H'=\{x' \mid x\in H\}$ a disjoint copy of $H$ and fix a distinguished element $\xi^{3k}$ of $H$. On $A=H\cup H'$ define the operation $*$ by
\[x*y=xy, \quad x'*y=x*y'=(xy)', \quad x'*y'=xy\xi^{3k}.\]
Then $(A,*)$ is an abelian group. Notice that different choices of $\xi^{3k}$
may give rise to nonisomorphic groups; for example if $m=|H|$ is even, then $k=0$ gives $C_m\times C_2$ and $k=1$ gives $C_{2m}$.

We are in a position to construct a light dual multinet $\Lambda$ labeled by $(A,*)$ using the maps $\alpha_1,\alpha_2,\alpha_3:G\to \PG(2,\K)$ by
\[\alpha_1: \left\{ \begin{array}{l} x\mapsto f_1(x) \\ x' \mapsto f_2(x^{-1}) \end{array} \right.
\quad
\alpha_2: \left\{ \begin{array}{l} y\mapsto f_1(y\xi) \\ y' \mapsto f_2(y^{-1}\xi^{-1}) \end{array} \right.
\quad
\alpha_3: \left\{ \begin{array}{l} z\mapsto f_1(z^{-1}\xi^{3k-1}) \\ z' \mapsto f_2(z\xi) \end{array} \right.
\]
where $x,y,z \in H$. Again, the computation is straightforward, and $\ell$ is the unique line of $\Lambda$ of length $m=n/2>1$. For example:
\begin{align*}
\alpha_1(x'), \alpha_2(y'), \alpha_3(z) \text{ are collinear} & \Leftrightarrow f_2(x^{-1}), f_2(y^{-1}\xi^{-1}), f_1(z^{-1}\xi^{3k-1}) \text{ are collinear}\\
&\Leftrightarrow x^{-1}y^{-1}\xi^{-1}=z^{-1}\xi^{3k-1}\\
&\Leftrightarrow xy\xi^{3k}= z \\
&\Leftrightarrow x'*y'=z.
\end{align*}

\begin{remark}
\label{rem1}
{\emph{
This example shows that, in contrast to dual $3$-nets, the isotopy class of the labeling quasigroup is not uniquely determined for light dual multinets.}}
\end{remark}

\subsection{Tetrahedron type light dual multinets} One can construct a tetrahedron type dual $3$-net in $\PG(3,\K)$ using the maps
\[\alpha_1: \left\{ \begin{array}{l} x\mapsto (x,0,1,0) \\ x\sigma \mapsto (0,1,0,x) \end{array} \right.
\quad
\alpha_2: \left\{ \begin{array}{l} y\mapsto (1,y,0,0) \\ y\sigma \mapsto (0,0,1,y) \end{array} \right.
\quad
\alpha_3: \left\{ \begin{array}{l} z\mapsto (0,-z,1,0) \\ z\sigma \mapsto (1,0,0,-z) \end{array} \right.
\]
where $x,y,z\in H\subseteq \K^*$, $|H|=m$ and $\sigma$ is an element of the semidirect product $H\rtimes \langle \sigma \rangle$ satisfying $\sigma^2=(x\sigma)^2=1$ for all $x\in H$. In other words, $G=H\rtimes \langle \sigma \rangle$ is isomorphic to the dihedral group of order $n=2m$; see \cite[Section 4.4]{knp_3}, \cite[Section 6.2]{per}. By composing the $\alpha_i$'s with a projection from a point $P$ in general position to a fixed hyperplane $\Sigma$, one obtains a dual 3-net labeled by $D_{2m}$ in $\PG(2,\K)$. As observed in \cite{BartzYuz2014} (in the primal setting), if the center $P$ of the projection is a generic point of one of the planes $X_i=0$ ($i=1,2,3,4$), then we obtain a planar light dual multinet $\Lambda$. $\Lambda$ is contained in the union of four lines, three of which are concurrent. Hence, $\Lambda$ is not algebraic. Also,  $\Lambda$ has exactly one line of length $n/2>1$.

\subsection{Light dual multinet of order $18$ with lines of length $3$} \label{sec:example18}
Let $\xi$ be a $9$th root of unity in $\K$ and $\omega=\xi^3$ a $3$rd root of unity. We define the sets $\Lambda_1$, $\Lambda_2$, $\Lambda_3$ of projective point by the following $18 \times 3$ matrices.

\[
\Lambda_1=\begin{bmatrix}0 & 1 & 1 \\
0 & \omega^2 & 1 \\
0 & \omega & 1 \\
-\xi^8 & 1 & 0 \\
-\xi^2 & 1 & 0 \\
-\xi^5 & 1 & 0 \\
\xi^2 & 0 & 1 \\
\xi^8 & 0 & 1 \\
\xi^5 & 0 & 1 \\
-\xi^2 & -1 & 1 \\
-\xi^8 & -\omega & 1 \\
-\xi^5 & -\omega^2 & 1 \\
-\xi^2 & -\omega^2 & 1 \\
-\xi^8 & -1 & 1 \\
-\xi^5 & -\omega & 1 \\
-\xi^2 & -\omega & 1 \\
-\xi^8 & -\omega^2 & 1 \\
-\xi^5 & -1 & 1 \\
\end{bmatrix}
\quad
\Lambda_2=\begin{bmatrix}
0 & \xi^2 & 1 \\
0 & \xi^8 & 1 \\
0 & \xi^5 & 1 \\
-\xi & 1 & 0 \\
-\xi^4 & 1 & 0 \\
-\xi^7 & 1 & 0 \\
1 & 0 & 1 \\
\omega^2 & 0 & 1 \\
\omega & 0 & 1 \\
-1 & -\xi^5 & 1 \\
-\omega^2 & -\xi^8 & 1 \\
-\omega & -\xi^2 & 1 \\
-1 & -\xi^2 & 1 \\
-\omega^2 & -\xi^5 & 1 \\
-\omega & -\xi^8 & 1 \\
-1 & -\xi^8 & 1 \\
-\omega^2 & -\xi^2 & 1 \\
-\omega & -\xi^5 & 1 \\
\end{bmatrix}
\quad
\Lambda_3=\begin{bmatrix}
0 & \xi^4 & 1 \\
0 & \xi & 1 \\
0 & \xi^7 & 1 \\
\xi & 0 & 1 \\
\xi^4 & 0 & 1 \\
\xi^7 & 0 & 1 \\
-1 & 1 & 0 \\
-\omega^2 & 1 & 0 \\
-\omega & 1 & 0 \\
-\xi^4 & -\xi^4 & 1 \\
-\xi^7 & -\xi & 1 \\
-\xi & -\xi^7 & 1 \\
-\xi^7 & -\xi^7 & 1 \\
-\xi & -\xi^4 & 1 \\
-\xi^4 & -\xi & 1 \\
-\xi & -\xi & 1 \\
-\xi^4 & -\xi^7 & 1 \\
-\xi^7 & -\xi^4 & 1 \\
\end{bmatrix}
\]

One can check that $(\Lambda_1,\Lambda_2,\Lambda_3)$ is a light dual multinet of order $18$. The lines $X_1=0$, $X_2=0$, $X_3=0$ have length $3$, all other lines of the dual multinets have length $1$. Together with the labeling by the row numbers of $\Lambda_1,\Lambda_2,\Lambda_3$, the lines of length $1$ result the following partial latin square $S$ of order $18$.

\begin{eqnarray*}
\begin{array}{ccccccccc|ccccccccc}
* & * & * & 4 & 5 & 6 & 7 & 8 & 9 & 10 & 11 & 12 & 13 & 14 & 15 & 16 & 17 & 18 \\
* & * & * & 6 & 4 & 5 & 9 & 7 & 8 & 12 & 10 & 11 & 15 & 13 & 14 & 18 & 16 & 17 \\
* & * & * & 5 & 6 & 4 & 8 & 9 & 7 & 11 & 12 & 10 & 14 & 15 & 13 & 17 & 18 & 16 \\
4 & 6 & 5 & * & * & * & 2 & 3 & 1 & 15 & 13 & 14 & 16 & 17 & 18 & 11 & 12 & 10 \\
5 & 4 & 6 & * & * & * & 3 & 1 & 2 & 13 & 14 & 15 & 17 & 18 & 16 & 12 & 10 & 11 \\
6 & 5 & 4 & * & * & * & 1 & 2 & 3 & 14 & 15 & 13 & 18 & 16 & 17 & 10 & 11 & 12 \\
7 & 9 & 8 & 2 & 3 & 1 & * & * & * & 17 & 18 & 16 & 10 & 11 & 12 & 15 & 13 & 14 \\
8 & 7 & 9 & 3 & 1 & 2 & * & * & * & 18 & 16 & 17 & 11 & 12 & 10 & 13 & 14 & 15 \\
9 & 8 & 7 & 1 & 2 & 3 & * & * & * & 16 & 17 & 18 & 12 & 10 & 11 & 14 & 15 & 13 \\ \hline
10 & 12 & 11 & 16 & 17 & 18 & 13 & 14 & 15 & 3 & 1 & 2 & 7 & 8 & 9 & 4 & 5 & 6 \\
11 & 10 & 12 & 17 & 18 & 16 & 14 & 15 & 13 & 1 & 2 & 3 & 8 & 9 & 7 & 5 & 6 & 4 \\
12 & 11 & 10 & 18 & 16 & 17 & 15 & 13 & 14 & 2 & 3 & 1 & 9 & 7 & 8 & 6 & 4 & 5 \\
13 & 15 & 14 & 11 & 12 & 10 & 18 & 16 & 17 & 4 & 5 & 6 & 1 & 2 & 3 & 9 & 7 & 8 \\
14 & 13 & 15 & 12 & 10 & 11 & 16 & 17 & 18 & 5 & 6 & 4 & 2 & 3 & 1 & 7 & 8 & 9 \\
15 & 14 & 13 & 10 & 11 & 12 & 17 & 18 & 16 & 6 & 4 & 5 & 3 & 1 & 2 & 8 & 9 & 7 \\
16 & 18 & 17 & 15 & 13 & 14 & 11 & 12 & 10 & 8 & 9 & 7 & 4 & 5 & 6 & 2 & 3 & 1 \\
17 & 16 & 18 & 13 & 14 & 15 & 12 & 10 & 11 & 9 & 7 & 8 & 5 & 6 & 4 & 3 & 1 & 2 \\
18 & 17 & 16 & 14 & 15 & 13 & 10 & 11 & 12 & 7 & 8 & 9 & 6 & 4 & 5 & 1 & 2 & 3 \\
\end{array}
\end{eqnarray*}

This square can be extended into a latin square in many different ways, hence there are many non isomorphic quasigroups which can label $(\Lambda_1,\Lambda_2,\Lambda_3)$. Let $Q(\circ)$ be an arbitrary labeling quasigroup and assume that $(u,v,w)$ is an isotopism from $Q(\circ)$ to the group $G(\cdot)$. This means that for the left multiplication maps
\[L^\circ_x w = vL^\cdot_{u(x)}\]
holds. In particular, for all $x,y\in Q$, we have
\[L^\circ_x(L^\circ_y)^{-1}=vL^\cdot_{u(x)}(L^\cdot_{u(x)})^{-1}v^{-1}.\]
This shows that for any subset $U$ of $Q$, the permutation group generated by the set
\[\{ L^\circ_x(L^\circ_y)^{-1} \mid x,y\in U\}\]
is isomorphic to a subgroup of $G$.

Now, the left multiplication map $L_x^\circ$ is conjugate to the permutation $\rho_x$ determined by row $x$ of the table $S$. The permutation group $T$ generated by $\{\rho_x\rho_{x}^{-1}\mid x,y=10,\ldots,18\}$ has order $27$. This shows that $(\Lambda_1,\Lambda_2,\Lambda_3)$ cannot be labeled by a group.

\section{Light dual multinets labeled by cyclic groups}

For the rest of the paper, $Q$ is replaced by $G$ and it denotes a group of order $n$. Let $(\alpha_1,\alpha_2,\alpha_3)$ be a labeling $G\to \PG(2,\K)$ of the light dual multinet $\Lambda=(\Lambda_1,\Lambda_2,\Lambda_3)$, where $\Lambda_i=\alpha_i(G)$, $i=1,2,3$. Let $\ell$ be a line of maximal length, belonging to $(\Lambda_1,\Lambda_2,\Lambda_3)$ and assume that the labeling is such that $\Lambda_i\cap \ell = \alpha_i(H)$ for the subgroup $H$ of $G$. In order to avoid the trivial case, we assume $G\neq H$, that is, our dual multinet is not contained in a line.

The following lemma lists up some immediate observations.
\begin{lemma} \label{lm:Hiscyclic}
\begin{enumerate}[(i)]
\item If the points $\alpha_1(g), \alpha_2(g), \alpha_3(g)$ are contained in a line $m$, then $\alpha_1(1),\alpha_2(1) \in m$ as well, and $m=\ell$, $g\in H$ holds.
\item For $g\in G\setminus H$ we have the dual subnets \[(\alpha_1(H),\alpha_2(Hg),\alpha_3(Hg)), \quad
(\alpha_1(gH),\alpha_2(H),\alpha_3(gH)), \quad
(\alpha_1(Hg),\alpha_2(g^{-1}H),\alpha_3(H))\]
labeled by $H$.
\item If $H$ is cyclic and $g\in G\setminus H$ then $\alpha_2(Hg) \cup \alpha_3(Hg)$ is contained in a conic.
\item If $p=0$ or $p>n$ then then $H$ is cyclic and $\alpha_2(Hg) \cup \alpha_3(Hg)$ is contained in a conic.
\end{enumerate}
\end{lemma}
\begin{proof}
(i) and (ii) are trivial. (iii) and (iv) follow from Proposition \ref{pr:3net-properties}(v) and (iv), respectively.
\end{proof}

\begin{lemma} \label{lm:concurrent_lines}
Assume that $H$ is normal in $G$, $|H|\geq 3$, $x\in G\setminus H$, and that $\alpha_1(Hx)$ is contained in a line $m_1$. Then the following hold:
\begin{enumerate}[(i)]
\item $\alpha_2(Hx), \alpha_3(Hx)$ are contained in the lines $m_2,m_3$ and $m_1,m_2,m_3$ have a point $Q\not\in \ell$ in common.
\item If the lines $m_1,m_2,m_3$ are distinct then $x^2\in H$ holds for all $x\in G$.
\item If there is a cyclic collineation group $\mathcal{H}$ with orbits $\alpha_1(H)$, $\alpha_2(H)$, $\alpha_3(H)$ then $m_1=m_2$.
\end{enumerate}
\end{lemma}
\begin{proof}
(i) Apply Proposition \ref{pr:triangles} for the triangular dual $3$-nets $(\alpha_1(H),\alpha_2(Hx),\alpha_3(Hx))$ and $(\alpha_1(Hx),\alpha_2(H),\alpha_3(Hx))$ to obtain $\alpha_2(Hx) \subseteq m_2$, $\alpha_3(Hx)\subseteq m_3$, and $m_1\cap m_2 \in m_3$.

(ii) Assume that $m_1,m_2,m_3$ are distinct; let $Q$ be their common point and $T_i=\ell \cap m_i$, $i=1,2,3$. Let $\mathcal{H}_1, \mathcal{H}_2$ be the collineation groups of the above dual subnets, having $\alpha_1(Hx)$ and $\alpha_2(Hx)$ as orbits, respectively. The fixed points of $\mathcal{H}_i$ on $m_i$ are $Q$ and $T_i$, $i=1,2$. Consider the triangular dual subnet $(\alpha_1(Hx),\alpha_2(Hx),\alpha_3(Hx^2))$. The action of the corresponding collineation group on $m_i$ coincides with the actions of $\mathcal{H}_i$; see \cite[Proposition 10]{knp_3}. Thus, the line containing $\alpha_3(Hx^2)$ passes through $T_1,T_2$. This means that $\alpha_3(Hx^2)\subseteq \ell$, and $H=Hx^2$.

(iii) Assume now that $m_1,m_2,m_3$ are distinct and that there is a cyclic collineation group $\mathcal{H}$ with orbits $\alpha_1(H)$, $\alpha_2(H)$. From Proposition \ref{pr:3net-properties}(i),  there is a cyclic collineation group $\mathcal{U}$ fixing the components of the the triangular dual $3$-nets $(\alpha_1(H),\alpha_2(Hx),\alpha_3(Hx))$. Proposition \ref{pr:triangles} implies that $\mathcal{H}$ and $\mathcal{U}$ have the same action on $\ell$. In particular, $\mathcal{H}$ has two fixed points $S,T$ on $\ell$, and $\{S,T\}$ is contained in $m_2\cup m_3$. Similarly one shows that $\{S,T\} \subseteq m_1\cup m_3, m_1\cup m_2$ which is not possible if the $m_i$'s are distinct.
\end{proof}

\begin{proposition} \label{pr:index2class}
If $H$ is a cyclic subgroup of index $2$ in $G$ then $\Lambda=(\Lambda_2,\Lambda_2,\Lambda_3)$ is either of conic-line or of tetrahedron type. If $n=4$ then these two types are the same.
\end{proposition}
\begin{proof}
Case 1: $n=4$. Write $H=\{1,h\}$ and $g\in G\setminus H$. We have
\begin{alignat*}{2}
\alpha_1(g)\alpha_2(gh) \cap \alpha_1(gh)\alpha_2(g)&=\alpha_3(g^2h) &\in
 \ell, \\
\alpha_1(g)\alpha_3(gh) \cap \alpha_1(gh)\alpha_3(g)&=\alpha_2(h) &\in \ell, \\
\alpha_2(g)\alpha_3(gh) \cap \alpha_2(gh)\alpha_3(g)&=\alpha_1(h) &\in \ell. \\
\end{alignat*}
On the one hand, Pascal's Theorem implies that $\alpha_1(Hg)\cup \alpha_2(Hg) \cup \alpha_3(Hg)$ is contained in a conic. On the other hand, the triangles $\alpha_1(g) \alpha_2(g) \alpha_3(g)$ and $\alpha_1(gh)\alpha_2(gh)\alpha_3(gh)$ are in perspective position with respect to the line $\ell$:
\begin{alignat*}{2}
\alpha_1(g)\alpha_2(g) \cap \alpha_1(gh)\alpha_2(gh)&=\alpha_3(g^2) &\in
 \ell, \\
\alpha_1(g)\alpha_3(g) \cap \alpha_1(gh)\alpha_3(gh)&=\alpha_2(1) &\in \ell, \\
\alpha_2(g)\alpha_3(g) \cap \alpha_2(gh)\alpha_3(gh)&=\alpha_1(1) &\in \ell. \\
\end{alignat*}
Thus by Desargues' Theorem, the lines $\alpha_1(g)\alpha_1(gh)$, $\alpha_2(g)\alpha_2(gh)$, $\alpha_3(g)\alpha_3(gh)$ are concurrent.

Case 2: $n\geq 6$ and $\alpha_i(Hg)$ is contained in a line $m_i$ for $i=1,2,3$. By Lemma \ref{lm:concurrent_lines}, the $m_i$'s share a point. Hence, $\Lambda$ is either a tetrahedron type dual multinet, or, $m_i=m_j$ for some $i\neq j$. As $H$ has index $2$ in $G$, the latter would imply $m_1=m_2=m_3=\ell$, a contradiction.

Case 3: $n\geq 6$ and $\alpha_1(Hg)$ is not contained in a line. By Lemma \ref{lm:Hiscyclic}(iii), $\Delta=(\alpha_1(Hg),\alpha_2(H),\alpha_3(Hg))$ and $\Delta'=(\alpha_1(Hg),\alpha_2(Hg),\alpha_3(H))$ are conic-line type dual $3$-nets. Let $C,C'$ be the corresponding irreducible conics. The subgroup $H$ is cyclic, $H=\langle h \rangle$. Proposition \ref{pr:3net-properties}(iii) implies the existence of collineations $\varphi,\varphi'$ such that $\varphi,\varphi'$ preserve $\Delta,\Delta'$, respectively. Moreover, both $\varphi$ and $\varphi'$ map $\alpha_1(h^ig)$ to $\alpha_1(h^{i+1}g)$ and fix $\ell$ and the points of $\ell \cap C, \ell \cap C'$. As $n/2\geq 3$, we have $\varphi=\varphi'$, which implies $\ell\cap C_1=\ell\cap C_2$ equals the set of fixed points of $\varphi$ on $\ell$. Hence, $|C\cap C'|\geq 5$ and $C=C'$. In particular, $\ell \cup C$ contains $\Lambda$, showing that $\Lambda$ is of conic-line type.
\end{proof}

\begin{lemma} \label{lm:tetrahedral_dihedral}
If $\Lambda$ is of tetrahedron type then $ghg^{-1}=h^{-1}$ holds for any $h\in H$, $g\in G\setminus H$.
\end{lemma}
\begin{proof}
Assume that $\Lambda$ is contained in the lines $\ell,m_1,m_2,m_3$ and be $P$ the common point of $m_1,m_2,m_3$. In this case $|G:H|=2$ and $H$ is cyclic. For any $g_1,g_2\in G\setminus H$ the triangles $\alpha_1(g_i)\alpha_2(g_i)\alpha_3(g_i)$ ($i=1,2$) are in perspective position from $P$, which implies that the points $\alpha_1(g_1)\alpha_2(g_1) \cap \alpha_1(g_2)\alpha_2(g_2)$, $\alpha_1(g_1)\alpha_3(g_1) \cap \alpha_1(g_2)\alpha_3(g_2)=\alpha_2(1)$, $\alpha_2(g_1)\alpha_3(g_1) \cap \alpha_2(g_2)\alpha_3(g_2)=\alpha_1(1)$ are collinear. Hence, $\alpha_1(g_1)\alpha_2(g_1)$ and $\alpha_1(g_2)\alpha_2(g_2)$ intersect in a common point $\alpha_3(h_0)$ of $\ell$. Therefore, $g_1^2=g_2^2=h_0$. In particular, $g^2=(gh)^2$, which implies $ghg^{-1}=h^{-1}$.
\end{proof}

\begin{proposition} \label{pr:cyclic_alg_net}
Finite light dual multinets labeled by a cyclic group are algebraic.
\end{proposition}
\begin{proof}
Using the notation of the rest of this section, we assume that $G$ is a cyclic group. If $|G:H|\leq 2$ then the claim follows from Proposition \ref{pr:index2class} and Lemma \ref{lm:tetrahedral_dihedral}. Assume $|G:H|\geq 3$ and $n\geq 6$. Identify $G$ with $\mathbb{Z}/n\mathbb{Z}$ and use the additive notation.  We claim that the proof of \cite[Proposition 18]{knp_3} can be used to prove that $\Lambda$ is contained in a cubic curve $\mathcal{F}$. That proof is based on Lam\'e configurations consisting of nine points
\[\alpha_1(x_1), \alpha_2(y_1), \alpha_3(z_1),\quad \alpha_1(x_2), \alpha_2(y_2), \alpha_3(z_2), \quad \alpha_1(x_3), \alpha_2(y_3), \alpha_3(z_3)\]
with $x_i,y_i,z_i\in G$, $i=1,2,3$. These points form six triples contained in the six lines $r_1,\ldots,r_6$. One can check that

\begin{enumerate}
\item[(*)] for each Lam\'e configuration $\mathcal{L}$ given in the proof of \cite[Proposition 18]{knp_3}, we have $|x_i-x_j|, |y_i-y_j|, |z_i-z_j|\leq 2$.
\end{enumerate}

In order to make the proof work for the multinet case, we have to show that for the nine points and the six lines of $\mathcal{L}$ are distinct. For the nine points this is clear, since the multinet $\Lambda$ is light. Assume now that $r_k=r_{k'}$, $k\neq k'$ for one of the Lam\'e configurations. Then, $r_k$ is a long line of $\Lambda$ which contains the points $\alpha_i(u)$, $\alpha_i(u')$ for some $i\in \{1,2,3\}$ and $u,u' \in G$, $u\neq u'$. This means that $u,u'$ are contained in a coset of a subgroup $H_1$ of $G$, corresponding to a long line of $\Lambda$. Therefore $|u-u'|\in H_1$, and (*) implies that $|G:H_1|\leq 2$. This contradicts to the assumption that the longest line of $\Lambda$ has length $\leq n/3$.
\end{proof}

\section{Uniqueness results}

In this section we prove that sufficiently large light dual multinets of triangular, conic-line or tetrahedron type cannot be extended in the following sense.

\begin{definition} \label{def:noext}
Let $Q,R$ be quasigroups, $\Lambda=(\Lambda_1,\Lambda_2,\Lambda_3)$, $\Delta=(\Delta_1,\Delta_2,\Delta_3)$ dual multinets labeled by $Q$ and $R$. Let $(\alpha_1,\alpha_2,\alpha_3)$ and $(\beta_1,\beta_2,\beta_3)$ be the respective labelings. We say that $\Lambda$ is a \emph{dual submultinet} of $\Delta$, if $Q\leq R$ and $\alpha_i=\beta_i|_Q$ for $i=1,2,3$.

We say that a group-labeled light dual multinet $\Lambda$ \emph{cannot be extended} if it is not the proper dual submultinet of a group-labeled light dual multinet $\Delta$.
\end{definition}

We continue to use the notation of the preceding section.

\begin{lemma}\label{lm:triang_index3}
Assume that the light dual multinet $(\Lambda_1,\Lambda_2,\Lambda_3)$ is contained in a triangle $\ell=\ell_1$, $\ell_2$, $\ell_3$, and $|H|\geq 3$. Then $H$ is a cyclic normal subgroup of index $3$ in $G$. Moreover, for all $i=1,2,3$ and $x\in G$, there is a cyclic collineation group $\mathcal{H}$ fixing the corners $\ell_i\cap \ell_j$ of the triangle and having $\alpha_i(Hx)$ as an orbit.
\end{lemma}
\begin{proof}
For any $x\in G\setminus H$, we can project $\alpha_1(H)$ from $\alpha_2(x)$ to $\ell_3$. Hence, for all $i=1,2,3$, $\alpha_i(Hx)$ is contained in one of the three lines $\ell_1,\ell_2,\ell_3$. It is immediate that $\ell_1,\ell_2,\ell_3$ are lines of the same length $|H|$. This implies $|G:H|=3$ and $H$ must be cyclic. Moreover, for any $x\in G\setminus H$, $(\alpha_1(xH),\alpha_2(Hx^{-1}),\alpha_3(xHx^{-1}))$ is either a triangular dual subnet, or contained in a line. Since $\alpha_3(xHx^{-1})\cap \alpha_3(H)\neq \emptyset$, we have $\alpha_3(xHx^{-1})\subseteq \ell_1$ and $xHx^{-1}=H$. The existence of the collineation group $\mathcal{H}$ follows from the fact that each $\alpha_i(Hx)$ is contained in a  triangular dual subnet with sides $\ell_1,\ell_2,\ell_3$.
\end{proof}

\begin{lemma}\label{lm:triang_noext}
If $|H|\geq 4$ then triangle type light dual multinets cannot be extended in the sense of Definition \ref{def:noext}.
\end{lemma}
\begin{proof}
Let $G_1$ be a group and assume that $\Lambda$ is a proper dual submultinet of a $G_1$-labeled light dual multinet. Take elements $x\in G\setminus H$, $y\in G_1\setminus G$. By Lemma \ref{lm:triang_index3}, $H\triangleleft G$ and $G=H\cup Hx\cup Hx^2$. Let $\ell,\ell_1,\ell_2$ be the lines of the dual submultinet, and $T=\ell_1\cap \ell_2$, $T_i=\ell \cap \ell_i$, $i=1,2$. As $H$ is cyclic we have two subcases by Lemma \ref{lm:Hiscyclic}(iii).

Case 1: $\alpha_1(yH)$ is contained in a line $m$. Consider the dual submultinets
\begin{align*}
\Delta^{(0)}&=(\alpha_1(yH),\alpha_2(H),\alpha_3(yH))\\
\Delta^{(1)}&=(\alpha_1(yH),\alpha_2(Hx),\alpha_3(yHx))\\
\Delta^{(2)}&=(\alpha_1(yH),\alpha_2(Hx^2),\alpha_3(yHx^2))
\end{align*}
All three are triangular. By \cite[Proposition 10]{knp_3}, $m$ passes through the corners $T,T_1,T_2$, a contradiction.

Case 2: $\alpha_1(yH)$ is contained in the nonsingular conic $C$. Consider again the dual subnets $\Delta^{(0)}$, $\Delta^{(1)}$, $\Delta^{(2)}$. By Proposition \ref{pr:3net-properties}(iii), we have cyclic collineations $\varphi_0, \varphi_1, \varphi_2$, whose action on $\alpha_1(yH)$ coincide. As $|H|\geq 4$, we have $\varphi_0=\varphi_1=\varphi_2$. Moreover, $\ell$, $\ell_1$, $\ell_2$ pass through the two points of $C$, left fixed by $\varphi_0$. Thus, $\ell=\ell_1=\ell_2$, a contradiction.
\end{proof}

\begin{remark}
The light dual multinet of order $18$ given in Section \ref{sec:example18} shows that in Lemma \ref{lm:triang_noext}, the condition $|H|\geq 4$ is indeed needed. It is true that the example of order $18$ is not labeled by a group, but its subloop structure is such that our method cannot handle the group-labeled extensions of its triangular dual submultinet.
\end{remark}

\begin{lemma}\label{lm:conicline_index2}
Assume that the light dual multinet $(\Lambda_1,\Lambda_2,\Lambda_3)$ is contained in $\ell \cup C$ where $C$ is a nonsingular conic, and $|H|\geq 5$. Then $H$ is a cyclic normal subgroup of index $2$ in $G$. Moreover, there is a cyclic collineation group $\mathcal{H}$ fixing $\ell\cup C$ and having $\alpha_i(Hx)$ as an orbit for all $i=1,2,3$, $x\in G$.
\end{lemma}
\begin{proof}
By assumption, $\cup_{i=1}^3 \alpha_i(G\setminus H)$ is contained in $C$. Fix elements $x,y\in G\setminus H$ and project $\alpha_2(yH)$ from $\alpha_1(x)$ to $\alpha_3(xyH)$. As the latter is disjoint from $C$, we have $\alpha_3(H)=\alpha_3(xyH)$. Thus, $xy\in H$, which implies $|G:H|=2$ and $H\triangleleft G$.

By looking at the dual submultinets
\begin{align*}
&(\alpha_1(H), \alpha_2(Hx), \alpha_3(Hx)),\\
&(\alpha_1(Hx), \alpha_2(H), \alpha_3(Hx)),\\
&(\alpha_1(Hx), \alpha_2(Hx), \alpha_3(H)),
\end{align*}
one obtains cyclic collineation groups $\mathcal{H}_1$, $\mathcal{H}_2$, $\mathcal{H}_3$, whose actions coincide on $\alpha_i(Hx)$, consisting of at least 5 points in general position. Hence, $\mathcal{H}_1=\mathcal{H}_2=\mathcal{H}_3$.
\end{proof}

\begin{lemma}\label{lm:conicline_noext}
If $|H|\geq 7$ then conic-line type dual light multinets cannot be extended in the sense of Definition \ref{def:noext}.
\end{lemma}
\begin{proof}
Let $G_1$ be a group and assume that $\Lambda$ is a proper dual submultinet of a $G_1$-labeled light dual multinet. Take elements $x\in G\setminus H$, $y\in G_1\setminus G$. Assume $\cup_{i=1}^3 \alpha_i(G) \subseteq \ell \cup C$, $\cup_{i=1}^3 \alpha_i(H) \subseteq \ell$, and $\{T_1,T_2\}=\ell \cap C$. By Lemma \ref{lm:conicline_index2}, $H\triangleleft G$ and $G=H\cup Hx$. Moreover, there is a collineation group $\mathcal{H}=\langle \varphi \rangle$  fixing $\ell,C$ and having $\alpha_2(Hx)$ as an orbit.

Case 1: $\alpha_1(yH)$ is contained in the line $m$. Then, $\alpha_3(yH)$ is contained in the line $m'$ and $m,m'$ pass through $T_1,T_2$. The dual subnet $(\alpha_1(yH),\alpha_2(Hx),\alpha_3(yHx))$ is of conic-line type and $\alpha_2(Hx) \cup \alpha_3(yHx)$ is contained in a nonsingular conic $C'$. Moreover, by Proposition \ref{pr:3net-properties}, there is a collineation group $\mathcal{H'}=\langle \varphi' \rangle$ fixing $C'$ and $m$. Moreover, the actions of $\varphi$ and $\varphi'$ coincide on $\alpha_2(Hx)$, hence $\varphi=\varphi'$ and $m=\ell$, a contradiction.

Case 2: $\alpha_3(yHx)$ is contained in the line $m$. We can modify the proof for case 1 to eliminate this possibility.

Case 3: $\alpha_1(yH)$ and $\alpha_3(xHy)$ are not contained in a line. Then, the dual subnet $(\alpha_1(yH),\alpha_2(Hx),\alpha_3(yHx))$ is labeled by the cyclic group $H$. Moreover, it is neither triangular nor of conic-line type, hence is contained in an irreducible cubic curve $\Gamma$ by Proposition \ref{pr:cyclic_alg_net}. This is not possible, as $\alpha_2(Hx)\subseteq C\cap \Gamma$ and $|H|\geq 7$.
\end{proof}

\begin{lemma}\label{lm:tetrahedron_index2}
Assume that the light dual multinet $(\Lambda_1,\Lambda_2,\Lambda_3)$ is contained in $\ell \cup m_1 \cup m_2 \cup m_3$ where $m_1,m_2,m_3$ are distinct lines passing through a point $Q\not\in \ell$. If $|H|\geq 3$ then $H$ is a cyclic normal subgroup of index $2$ in $G$. Moreover, there are collineation groups $\mathcal{H}_i$, $i=1,2,3$, such that the following holds:
\begin{enumerate}[(i)]
\item $\alpha_i(H)$ is an orbit of $\mathcal{H}_i$.
\item $\alpha_j(G\setminus H), \alpha_k(G\setminus H)$ are orbits of $\mathcal{H}_i$, $j,k\neq i$.
\item $\mathcal{H}_i$ fixes $\ell,m_j,m_k$, $j,k\neq i$.
\end{enumerate}
\end{lemma}
\begin{proof}
In the usual way we show that for any $x,y\in G\setminus H$, $\alpha_3(xHy) \subseteq \ell$, which implies $xHy=H$. $|G:H|=2$ and $H\triangleleft G$ follows. The construction and the properties for the groups $\mathcal{H}_i$ follow from the properties of triangular dual subnets.
\end{proof}

\begin{lemma}\label{lm:tetrahedron_noext}
If $|H|\geq 5$ then tetrahedron type dual light multinets cannot be extended in the sense of Definition \ref{def:noext}.
\end{lemma}
\begin{proof}
We use the notation of Lemma \ref{lm:tetrahedron_index2}. Let $G_1$ be a group and assume that $\Lambda$ is a proper dual submultinet of a $G_1$-labeled light dual multinet. Take elements $x\in G\setminus H$, $y\in G_1\setminus G$. As $H$ is cyclic we have two subcases by Lemma \ref{lm:Hiscyclic}(iii).

Case 1: $\alpha_1(yH) \subseteq m'$ line. Then we have triangular dual subnets
\[(\alpha_1(yH),\alpha_2(H),\alpha_3(yH)) \quad \mbox{ and } \quad (\alpha_1(yH),\alpha_2(Hx),\alpha_3(yHx)). \]
This implies $T_1\in m'$ or $T_3\in m'$. From $m'\neq \ell$, $T_2\not\in m'$ we obtain $Q\in m'$, $m'=m_1$ or $m'=m_3$, a contradiction.

Case 2: $\alpha_1(yH) \subseteq C$ nonsingular conic $C$. Same subnets as before imply that $\ell$ and $m_2$ are the fixed secant of the collineation group associated to $\alpha_1(yH)$. Hence, $\ell=m_2$, a contradiction.
\end{proof}

\begin{lemma}\label{lm:abelian_triang_or_conicline}
If $H$ is a cyclic normal subgroup of $G$ with $|H|\geq 7$ then any $G$-labeled dual light multinet is either triangular or of conic-line type or of tetrahedron type. In particular, $|G:H|\leq 3$.
\end{lemma}
\begin{proof}
Fix an element $x\in G\setminus H$.

Case 1: $\alpha_1(Hx)$ is contained in a line $m$. Then by Lemma \ref{lm:concurrent_lines}, $\alpha_2(Hx), \alpha_3(Hx)$ are contained in lines $m',m''$ such that $m,m',m''$ are concurrent. Assume that these lines are distinct lines. Then, as before, $\langle H,x\rangle$ labels a tetrahedron type dual submultinet. By Lemma \ref{lm:tetrahedron_noext}, we have $G=\langle H,x\rangle=H\langle x \rangle$. Assume that $m,m',m''$ are not distinct. $m=m''$ would mean that $\alpha_2(H)\subseteq m$, which is not possible. Similarly, $m'\neq m''$. Hence, $m=m'$ and $\alpha_1(Hx) \cup \alpha_2(Hx) \cup \alpha_3(Hx^2) \subseteq m$. Analogously, there is a line $k$ containing $\alpha_1(Hx^{-1}) \cup \alpha_2(Hx^{-1}) \cup \alpha_3(Hx^{-2})$. Considering the fixed points of the triangular dual subnets contained in $\ell \cup m \cup k$, one shows that the cyclic collineations corresponding to $\alpha_1(Hx)$, $\alpha_2(Hx)$ and $\alpha_3(Hx)$ share the same fixed points. Then, the dual subnet $(\alpha_1(Hx),\alpha_2(H),\alpha_3(Hx))$ turns out to be contained in $\ell\cup m \cup k$, which implies $Hx=Hx^{-2}$. Hence, we have a triangular submultinet labeled by the subgroup $H\langle x \rangle$. Lemma \ref{lm:triang_noext} implies $G=H\langle x \rangle$ and $|G:H|=3$.

Case 2: $\alpha_1(Hx)$ is contained in the nonsingular conic $C$. The dual subnets
\[(\alpha_1(Hx),\alpha_2(H),\alpha_3(Hx)) \mbox{ and } (\alpha_1(H),\alpha_2(Hx),\alpha_3(Hx))\]
are of conic-line type. By $|H|\geq 5$, $\alpha_1(Hx) \cup \alpha_2(Hx) \cup \alpha_3(Hx) \subseteq C$. Since the dual subnet $(\alpha_1(Hx),\alpha_2(Hx),\alpha_3(Hx^2))$ is labeled by the cyclic group $H$, it is contained in a cubic curve $\Gamma$ which has at least $10$ points in common with $C$. Therefore, $C$ is a component of $\Gamma$ and $\alpha_3(Hx^2)$ is contained in a line. Using the fixed points of the corresponding cyclic collineations, we obtain $\alpha_3(Hx^2)\subseteq \ell$. This implies $Hx^2=H$ and we have a conic-line type dual submultinet labeled by $H\langle x\rangle$. Lemma \ref{lm:conicline_noext} implies $G=H\langle x\rangle$ and $|G:H|=2$.
\end{proof}

In fact, the last part of the proof shows more.

\begin{lemma} \label{lm:cyclicH_order5}
If $H$ is a cyclic normal subgroup of order $5$ in $G$, then $G/H$ is either of order $3$ or an elementary abelian $2$-group. \qed
\end{lemma}

\section{Light dual multinets and elementary abelian groups}

\begin{lemma} \label{lm:c2^3}
If $p\neq 2$, then no light dual multinet can be labeled by the elementary abelian group of order $8$.
\end{lemma}
\begin{proof}
Yuzvinsky proved in \cite[Theorem 4.2]{ys2004} that $C_2^3$ cannot label a dual $3$-net. The proof uses the theorem of Desargues and the theorem of complete quadrilaterals. Hence, if the characteristic of the underlying field is different from two, then the proof works for light dual multinet labelings as well.
\end{proof}

\begin{proposition}\label{pr:elemabelian}
Let $r$ be prime and $G$ an elementary abelian $r$-group of order $r^k$. Assume $p=0$ or $p>|G|$. Let $\Lambda=(\Lambda_1,\Lambda_2,\Lambda_3)$ be a light dual multinet in $PG(2,\mathbb{K})$, labeled by $G$. Then one of the following holds:
\begin{enumerate}[(i)]
\item $\Lambda$ is contained in a line.
\item $k\leq 2$ and $\Lambda$ is an algebraic dual $3$-net.
\item $r=k=2$ and $\Lambda$ is the light dual multinet given in Proposition \ref{pr:index2class}.
\item $r=3$ and $k\geq 2$.
\end{enumerate}
\end{proposition}
\begin{proof}
Let us assume that $\Lambda$ is not contained in a line. If $\Lambda$ is a dual $3$-net then $k\leq 2$ and $\Lambda$ is algebraic by \cite[Theorem 1]{knp_3}. Let us assume that $\Lambda$ is not a dual $3$-net, that is, $|H|>1$. As $H$ is cyclic by Lemma \ref{lm:Hiscyclic}, we have $|H|=r$ and $k\geq 2$ by the assumptions.

Assume $r=2$. Lemma \ref{lm:c2^3} implies $k\leq 2$. It remains to show that $|H|=r>3$ and $k\geq 2$ cannot occur. This follows from Lemma  \ref{lm:abelian_triang_or_conicline} when $r\geq 7$ and from Lemma  \ref{lm:cyclicH_order5} when $r=5$.
\end{proof}

\begin{lemma} \label{lm:onlylines}
Let $G$ be a finite group containing a normal subgroup $S$ of order $m\geq 3$. Let $(\Lambda_1,\Lambda_2,\Lambda_3)$ be a light dual multinet labeled by $G$. Assume that every dual submultinet of $(\Lambda_1,\Lambda_2,\Lambda_3)$ labeled by $S$ is a triangular dual $3$-net. Then $S$ is cyclic, and $|G:S|\leq 2$.
\end{lemma}
\begin{proof}
The proof of \cite[Proposition 23]{knp_3} can be applied to the light dual multinet case.
\end{proof}

\begin{proposition}\label{pr:Hisnormal}
Let $G$ be a group of order $n$, $\K$ an algebraically closed field of characteristic $p$ and assume $p=0$ or $p>n$. Let $(\alpha_1,\alpha_2,\alpha_3)$ be a labeling $G\to \PG(2,\K)$ of the light dual multinet $\Lambda=(\Lambda_1,\Lambda_2,\Lambda_3)$, where $\Lambda_i=\alpha_i(G)$, $i=1,2,3$. Let $\ell$ be a line of maximal length of $\Lambda$ and assume that $\Lambda_i\cap \ell = \alpha_i(H)$ for the subgroup $H$ of $G$. If $\Lambda$ has no lines of length $3$ and $|H|\geq 9$ then $H$ is normal in $G$.
\end{proposition}
\begin{proof}
In order to avoid triviality we assume that $\Lambda$ is not contained in a line, that is, $H\lneqq G$. By Lemma \ref{lm:Hiscyclic}(iv), $H$ is cyclic. If $N_G(H)\neq H$ then by Lemma \ref{lm:abelian_triang_or_conicline} we have a submultinet which is either triangle, tetrahedron or conic-line type, and which cannot be extended by Lemmas \ref{lm:triang_noext}, \ref{lm:conicline_noext} and \ref{lm:tetrahedron_noext}. Hence, $H\triangleleft G$.

Assume now $H=N_G(H)$, and that $H$ is maximal in $G$. Fix an element $x\in G\setminus H$.

Case 1: $\alpha_1(xH)$ is not contained in a line. Then, the dual 3-net
\[(\alpha_1(xH), \alpha_2(H), \alpha_3(xH))\]
is contained in $C\cup \ell$ where $C$ is an irreducible conic. The dual 3-net
\[(\alpha_1(xH), \alpha_2(Hx^{-1}), \alpha_3(xHx^{-1}))\]
is contained in a cubic $\Gamma$. As $|C\cap \Gamma|>6$, $\Gamma=C\cup m$ for some line $m$. Moreover, there is cyclic collineation group preserving $C$ and having fixed points $C\cap \ell$, $C\cap m$. This implies $\ell=m$ and $H=xHx^{-1}$, a contradiction.

Case 2: $\alpha_2(Hx^{-1})$ is not contained in a line. The same argument using the dual 3-nets $(\alpha_1(H), \alpha_2(Hx^{-1}), \alpha_3(Hx^{-1}))$ and $(\alpha_1(xH),\alpha_2(Hx^{-1}),\alpha_3(xHx^{-1}))$ leads to contradiction.

Case 3: For each $x\in G\setminus H$, $\alpha_1(xH)$ and $\alpha_2(Hx)$ are contained in lines. Then the dual subnet $(\alpha_1(xH), \alpha_2(Hx^{-1}), \alpha_3(xHx^{-1}))$ is triangular and $\alpha_3(xHx^{-1})$ is contained in a line $m$. As $m\neq \ell$, we have $xHx^{-1}\cap H=\{1\}$ for all $x\in G\setminus H$. In particular, $\mathrm{core}_G(H)=1$ and $G$ is a primitive permutation group on the cosets of $H$ in $G$. Herstein's Theorem \cite{Herstein} implies that $G$ is solvable and $G= S\rtimes H$, where $S$ is an elementary abelian normal subgroup of $G$. Since $H$ is cyclic, Horo\v{s}evski\u\i's Theorem \cite{Horosevskii} implies $|S|>|H|$, which means that no submultinet labeled by $S$ is contained in a line. Let $S$ have order $r^k$ with prime $r$.

If $r=2$ then $|S|\leq 4$ by Proposition \ref{pr:elemabelian}(iii), a contradiction. Thus, $r\geq 3$ and all dual submultinets labeled by $S$ are dual $3$-nets. This follows from Proposition \ref{pr:elemabelian}(ii) and (iii), using the non-existence of lines of length $3$ when $r=3$. If all dual subnets labeled by $S$ are triangular, then $|G:S|\leq 2$ by Lemma \ref{lm:onlylines}, which is impossible. Hence, there is an element $x\in S$ and $i\in \{1,2,3\}$ such that $\alpha_i(Sx)$ is not contained in a line. We can assume $i=1$ without loss of generality.

Let $\Gamma_y$ denote the cubic curve containing the dual $3$-net $(\alpha_1(Sx),\alpha_2(Sy),\alpha_3(Sxy))$, labeled by $S$. As $|\Gamma_1\cap \Gamma_y|\geq |S|>9$, either $\Gamma_1=\Gamma_y$ for all $y\in G$, or $\Gamma_1$ and $\Gamma_y$ have a common irreducible conic component $C$. In the latter case, $\alpha_1(Sx) \subset C$ and the line component of $\Gamma_y$ passes through the two fixed points of the corresponding cyclic collineation. This means that $\Gamma_1=\Gamma_y$ for all $y\in G$ also in this case. It follows that $\Lambda_2$ is contained in $\Gamma_1$, which contradicts the fact that $|\Lambda_2\cap \ell|\geq 9$. This is our final contradiction.
\end{proof}

\begin{proof}[Proof of Theorem \ref{th:main}]
Let $\ell$ be a line of maximal length of $\Lambda$ and assume that $\Lambda_i\cap \ell = \alpha_i(H)$ for the subgroup $H$ of $G$. Assume that neither $\Lambda$ is contained in a line, nor it has a line of length $3$. Then Proposition \ref{pr:Hisnormal} implies that $H$ is a cyclic normal subgroup of $G$. Lemma \ref{lm:abelian_triang_or_conicline} proves the theorem.
\end{proof}

\appendix
\section{GAP program for the light dual multinet of order 18}

\begin{scriptsize}
\begin{verbatim}
xi:=E(9); w:=xi^3;
Lam:=[
  [
    [0, 1, 1],           [0, w^2, 1],         [0, w, 1],
    [-xi^8, 1, 0],       [-xi^2, 1, 0],       [-xi^5, 1, 0],
    [xi^2, 0, 1],        [xi^8, 0, 1],        [xi^5, 0, 1],
    [-xi^2, -1, 1],      [-xi^8, -w, 1],      [-xi^5, -w^2, 1],
    [-xi^2, -w^2, 1],    [-xi^8, -1, 1],      [-xi^5, -w, 1],
    [-xi^2, -w, 1],      [-xi^8, -w^2, 1],    [-xi^5, -1, 1]
  ],[
    [0, xi^2, 1],        [0, xi^8, 1],        [0, xi^5, 1],
    [-xi, 1, 0],         [-xi^4, 1, 0],       [-xi^7, 1, 0],
    [1, 0, 1],           [w^2, 0, 1],         [w, 0, 1],
    [-1, -xi^5, 1],      [-w^2, -xi^8, 1],    [-w, -xi^2, 1],
    [-1, -xi^2, 1],      [-w^2, -xi^5, 1],    [-w, -xi^8, 1],
    [-1, -xi^8, 1],      [-w^2, -xi^2, 1],    [-w, -xi^5, 1]
  ],[
    [0, xi^4, 1],        [0, xi, 1],          [0, xi^7, 1],
    [xi, 0, 1],          [xi^4, 0, 1],        [xi^7, 0, 1],
    [-1, 1, 0],          [-w^2, 1, 0],        [-w, 1, 0],
    [-xi^4, -xi^4, 1],   [-xi^7, -xi, 1],     [-xi, -xi^7, 1],
    [-xi^7, -xi^7, 1],   [-xi, -xi^4, 1],     [-xi^4, -xi, 1],
    [-xi, -xi, 1],       [-xi^4, -xi^7, 1],   [-xi^7, -xi^4, 1]
  ]
];;

ls:=List([1..18],i->List([1..18],j->-1));;
for i in [1..18] do
  for j in [1..18] do
    for k in [1..18] do
      a:=DeterminantMat([Lam[1][i],Lam[2][j],Lam[3][k]]);
      if IsZero(a) then
        if ls[i][j]=-1 then ls[i][j]:=k; else ls[i][j]:=-2; fi;
      fi;
    od;
  od;
od;
Display(ls);

rho:=List([10..18],i->PermList(ls[i]));
Size(Group(rho/rho[1]));
\end{verbatim}
\end{scriptsize}

\section{Yuzvinsky's proof when $G=C_2^3$}

In this section we preprint the S. Yuzvinsky's proof \cite{ys2004} on the non-existence of dual $3$-nets labeled by the elementary abelian group of order $8$. We remark that the proof uses on the one hand the theorem of Desargues which works over arbitrary characteristic, and on the other hand the theorem of complete quadrilaterals which works over arbitrary characteristic different from $2$.

\begin{lemma}[\text{\cite[Lemma 4.1]{ys2004})}] \label{lm:yuz2004}
Suppose there are nine points $a_1,a_2,a_3$, $b_1,b_2,b_3$, $c_1$, $c_2$, $	c_3$ such that the following triples are collinear: $\{a_1,b_1,c_1\}$, $\{a_1,b_2,c_2\}$, $\{a_1,b_3,c_3\}$, $\{a_2,b_1,c_3\}$, $\{a_2,b_3,c_1\}$, $\{a_3,b_2,c_3\}$ and $\{a_3,b_3,c_2\}$. Then the lines $\{(a_2a_3), (b_1b_2), (c_1c_2)\}$ intersect at one point.
\end{lemma}
\begin{proof}
The triangles $b_1b_2c_3$ and $c_1c_2b_3$ form a Desargues configuration. More precisely, the lines passing through pairs of corresponding vertices intersect at $a_1$. Besides $(b_1c_3) \cap (c_1b_3)=\{a_2\}$ and $(b_2c_3)\cap (c_2b_3)=\{a_3\}$. Now the result follows from the second Desargues theorem.
\end{proof}

\begin{theorem}[\text{\cite[Theorem 4.2]{ys2004})}]
The group $\mathbb{Z}_2^3$ cannot be realized.
\end{theorem}
\begin{proof}
Suppose that $\mathbb{Z}_2^3$ is realized by three sets of points $\{a_i\}$, $\{b_i\}$ and $\{c_i\}$ where $i=1,2,\ldots,8$. Let us consider the matrix $C$ of the pairing generated by the collinearity relation. For that we enumerate the rows of this matrix by $a_i$ and the columns by $b_j$. The $(a_i,b_j)$ entry is the point $c_k$ lying on the line $(a_ib_j)$. Without loss of generality we can assume that
\[C=\begin{pmatrix} A&B\\B&A \end{pmatrix}\]
where
\[A=\begin{pmatrix}
c_1 & c_2 & c_3 & c_4 \\
c_2 & c_1 & c_4 & c_3 \\
c_3 & c_4 & c_1 & c_2 \\
c_4 & c_3 & c_2 & c_1
\end{pmatrix}\]
and
\[B=\begin{pmatrix}
c_5 & c_6 & c_7 & c_8 \\
c_6 & c_5 & c_8 & c_7 \\
c_7 & c_8 & c_5 & c_6 \\
c_8 & c_7 & c_6 & c_5
\end{pmatrix}.\]
Using Lemma \ref{lm:yuz2004} for rows $a_1, a_3, a_4$ and columns $b_1,b_2,b_3$, we find that the lines $\{(a_3a_4), (b_1b_2), (c_1c_2)\}$ intersect at a common point, say $x$. Unsing the lemma for rows $a_1,a_5,a_6$ and columns $b_1,b_2,b_5$, we find that $x \in (a_5a_6)$ and similarly $x\in (a_7a_8)$. Using the lemma again for rows $a_5,a_7,a_8$ and columns $b_5,b_6,b_7$, we have $x\in (b_5b_6) \cap (a_1a_2)$.

Now consider the quadrangle $a_1b_1b_2a_2$. We have $(a_1b_1)\cap (a_2b_2)=\{c_1\}$ and $(a_1b_2)\cap (a_2b_1) =\{c_2\}$. Thus by the diagonal theorem the points $x_1=(c_1c_2)\cap (a_1a_2)$ and $x_2=(c_1c_2)\cap (b_1b_2)$ form a harmonic pair for the pair $\{c_1,c_2\}$. By the inclusions in the previous paragraph, however, $x_1=x=x_2$. This contradiction completes the proof.
\end{proof}

\end{document}